\documentclass{amsart}
\usepackage{times,mathrsfs,array,amssymb,pb-diagram}
\usepackage[toc,page,title,titletoc,header]{appendix} 
\usepackage{multirow,float} 
\usepackage{xcolor}

\newcounter{lemma}
\newtheorem{Theorem}{Theorem}
\newtheorem{Lemma}[lemma]{Lemma}
\newtheorem{Corollary}[lemma]{Corollary}
\newtheorem{Proposition}[lemma]{Proposition}
\newtheorem{theorem}{Theorem}

\theoremstyle{definition}

\newtheorem{Remark}[lemma]{Remark}

\def\X{\mathfrak{X}}
\def\H{\mathbb H}

\def\cH{\mathcal H}
\def\GL{\mathrm{GL}}

\def\tr{\operatorname{tr}}
\def\nm{\operatorname{nr}}
\def\disc{\operatorname{disc}}

\def\C{\mathbb C}
\def\A{\mathscr A}
\def\F{\mathcal F}

\def\O{\mathcal O}
\def\o{\mathfrak o}

\def\Q{\mathbb Q}

\def\QQ{\mathcal Q}
\def\R{\mathbb R}
\def\X{\mathfrak X}

\def\XDN{X^D_0(N)}
\def\Z{\mathbb Z}

\def\mod{\  \mathrm{mod}\ }

\def\Im{\mathrm{Im\,}}
\def\gen#1{\langle #1\rangle}
\def\JS#1#2{\left(\frac{#1}{#2}\right)}

\def\SL{\mathrm{SL}}
\def\Sp{\mathrm{Sp}}

\def\CM{\mathrm{CM}}
\def\Gal{\mathrm{Gal}}
\def\wt{\widetilde}
\def\M#1#2#3#4{\begin{pmatrix}#1&#2\\#3&#4\end{pmatrix}}
\def\SM#1#2#3#4{\left(\begin{smallmatrix}#1&#2\\#3&#4\end{smallmatrix}
  \right)}
\def\abs#1{\left|#1\right|}
\def\pair#1,#2{\left(#1,#2\right)}

\def\Vol{\operatorname{Vol}}

\def\x{\times}

\def\braces #1{\left\{#1\right\}}

\begin{document}

\title{Class number relations arising from intersections of Shimura curves and Humbert surfaces}

\author{Jia-Wei Guo}
\address{Department of Mathematics, National Taiwan
University, Taipei, Taiwan 10617}
\email{jiaweiguo312@gmail.com}

\author{Yifan Yang}
\address{Department of Mathematics, National Taiwan
University and National Center for Theoretical Sciences, Taipei,
Taiwan 10617}
\email{yangyifan@ntu.edu.tw}

\begin{abstract} By considering the intersections of Shimura curves
  and Humbert surfaces on the Siegel modular threefold, we obtain new
  class number relations. The result is a higher-dimensional analogue
  of the classical Hurwitz-Kronecker class number relation.
\end{abstract}

\thanks{The authors would like to thank Professor Fernando Rodriguez
  Villegas for bringing this problem to their attention and for
  sharing his private notes with them. The authors would also like to
  thank Professor Takao Yamazaki and Professor Takuya Yamauchi for
  many fruitful discussions. The authors are partially supported by
  Grant 102-2115-M-009-001-MY4 of the Ministry of Science and
  Technology, Taiwan (R.O.C.).
}
\maketitle

\section{Introduction} As arithmetic and algebraic objects, modular
forms and their generalizations, such as harmonic Maass forms, weakly
holomorphic modular forms, etc., play a significant role in number
theory. In many instances, class numbers emerge naturally as Fourier
coefficients of modular forms. For example, suppose that
$\theta(z)=\sum_{n\in\Z}e^{2\pi inz}$ is the classical Jacobi theta
function. Then the coefficients of the weight 
$3/2$ modular form $\theta(z)^3=\sum_nr_3(n)e^{2\pi inz}$ can be
represented by certain sums of Hurwitz class numbers $H(m)$, the
counting function of $\SL(2,\Z)$-equivalence class of positive
definite binary quadratic forms of discriminant $-m$, weighted by the
sizes of automorphism groups of quadratic forms. Moreover,
Zagier \cite{Zagier-half-integral-weight} showed that with $H(0)$
defined to be $-1/12$, the function
$\sum_{n\in\Z}H(n)e^{2\pi inz}$ is the holomorphic part of a harmonic
Maass form of weight $3/2$.  
Around the same time, motivated by Dirichlet's class number formula,
Cohen \cite{Cohen-modularforms} generalized the definition of the
Hurwitz class numbers and define $H(r,n)$ via values of Dirichlet
$L$-function for quadratic character
$\left(\frac{n}{\cdot}\right)$. He showed that each power series
of the form $\sum^\infty_{n=0}H(r,n)e^{2\pi inz}$ is a modular form of
weight $r+1/2$ on $\Gamma_0(4)$, which is referred to as Cohen's
Eisenstein series of weight $r+1/2$.

On the other hand, because CM-points on a modular curve or a Shimura
curve $X$ correspond to abelian varieties with extra 
endomorphisms, when we embed $X$ into a modular variety of higher
dimension, CM-points often lie on intersections of $X$ with other
special cycles of the modular variety. By analyzing what CM-point lie
on the intersection and comparing with the intersection number, one
obtain relations among class numbers. For example, the classical
Hurwitz-Kronecker class number relation
\begin{equation} \label{equation: Kronecker}
\sum_{x\in\Z}H(4n-x^2)+\sum_{dd'=n,~d,d'>0}\min(d,d')
=\begin{cases}
  -1/12, &\text{if }n=0, \\
  2\sum_{d|n}d, &\text{if }n>0, \end{cases}
\end{equation}
can be interpreted as a relation arising from the
intersection of two certain curves on $X_0(1)\times X_0(1)$ defined by
modular correspondence (see \cite{Gross-Keating}). Note that the
generating function for the right-hand side of \eqref{equation:
  Kronecker} is the Eisenstein series $-E_2(z)/12$ of weight $2$,
which is the holomorphic part of a harmonic Maass form.
As another important example, Hirzebruch and Zagier
\cite{Zagier-intersection} considered a certain series of curves
$T_1,T_2,\ldots$ that are images of modular curves or Shimura curves
on the Hilbert modular surface $\SL(2,\o)\backslash\H^2$, where
$\o$ is the ring of integers in $\Q(\sqrt p)$, $p\equiv 1\mod 4$,
and obtained formulas involving Hurwitz class numbers for the
intersection numbers of $T_m$ and $T_n$. They also showed that if we fix $m$,
then the generating series for the intersection numbers is a modular
form of weight $2$ and Nebentype on $\Gamma_0(p)$. This results in
identities relating sums of class numbers to Fourier coefficients of a
modular form of weight $2$ and Nebentype.

Note that in the two aforementioned examples, the generating series
for intersection numbers are a modular form and the holomorphic part
of a harmonic Maass form, respectively. These arithmetic and
geometric phenomenon of modular forms has been further explored and
generalized by Kudla and Milson, from the aspect of automorphic 
representations, as correspondence between spaces of higher genus
modular forms and intersections of special cycles on the quotients of
symmetric spaces for orthogonal groups and the unitary groups
\cite{Kudla-intersection-cycles, Oda-modularforms}.

In this paper, we shall consider the case of Shimura curves on the
Siegel modular threefold
$\A_2:=\operatorname{Sp}(4,\Z)\backslash\H_2$.
By considering the intersections of a Shimura curve with Humbert
surfaces on $\A_2$, we will obtain identities relating sums of Hurwitz
numbers to the Fourier coefficients of Cohen's Eisenstein series of
weight $5/2$.
To state our theorem, let us first recall some definitions about
quadratic forms. Let $D_0$ be a squarefree integer, and $D_0=p_1\ldots
p_k$ be its prime factorization. Here we are concerned with positive
definite binary quadratic forms $Q$ of discriminants $-16D_0$ such
that all integers represented by $Q$ are congruent to $0$ or $1$
modulo $4$. Such a quadratic form is either primitive or of the form
$4Q'$ for some quadratic form $Q'$ of discriminant $-D_0$, and the
latter case occurs only when $D_0\equiv 3\mod 4$.
Let $C_{-16D_0}$ and $C_{-D_0}$ (in the case $D_0\equiv 3\mod 4$) be
the class group of primitive binary quadratic forms of discriminant
$-16D_0$ and $-D_0$, respectively.
Define characters $\chi_{-4}$ and $\chi_{p_j}$, $j=1,\ldots, k$, of
$C_{-16D_0}$ by
\begin{equation} \label{equation: genus characters}
\chi_{-4}:Q\mapsto\left(\frac{-4}{a}\right),\qquad
\chi_{p_j}:Q\mapsto\begin{cases}
  \displaystyle \left(\frac{a}{p_j}\right),\ &\text{if }p_j
  \text{ is odd},\\
  \displaystyle\left(\frac{8}{a}\right), &\text{if }p_j=2,\end{cases}
\end{equation}
where $a$ is any positive integer represented by $Q$ satisfying
$(a,2D_0)=1$. These characters generate the group of genus characters
of $C_{-16D_0}$ with a single relation that either
$\chi_{-4}\chi_{p_1}\ldots\chi_{p_k}$ or $\chi_{p_1}\ldots\chi_{p_k}$
is the trivial character depending on whether the odd part of $D_0$ is
congruent to $1$ or $3$ modulo $4$ (see \cite[Theorem 3.15.]{Cox}).
The condition that all integers represented by $Q$ are congruent
to $0$ or $1$ modulo $4$ means that $\chi_{-4}(Q)=1$ so that the
number of $p_j$ such that $\chi_{p_j}(Q)=-1$ is even.
Thus, we may associate to $Q$ an indefinite quaternion algebra over
$\Q$ of discriminant
\begin{equation} \label{equation: DQ}
D=D_Q=\prod_{p_j:~\chi_{p_j}(Q)=-1}p_j.
\end{equation}
Set also
\begin{equation} \label{equation: NQ}
N=N_Q=\prod_{p_j:~\chi_{p_j}(Q)=1}p_j.
\end{equation}
Let $\O$ be an Eichler order of level $N$ in the quaternion algebra of
discriminant $D$ over $\Q$, and let $X_0^D(N)$ be the Shimura curve
associated to $\O$.

Similarly, we may define $\chi_{p_j}$ on $C_{-D_0}$ and the product
$\chi_{p_1}\ldots\chi_{p_k}$ is the trivial character. Thus, for a
quadratic form $Q$ of discriminant $-16D_0$ of the form $4Q'$, we may
also associate to $Q$ a Shimura curve $X_0^D(N)$ with $D$ and $N$
given by \eqref{equation: DQ} and \eqref{equation: NQ}, respectively.

Now for a negative discriminant $d$, let $h_{D,N}(d)$ be the number of
CM-points of discriminant $d$ on $X_0^D(N)$. Write $d$ as $d=f^2d_0$,
where $d_0$ is a fundamental discriminant, and define
\begin{equation} \label{equation: HDN}
  H_{D,N}(|d|)=\frac1{2^{\omega_{D_0}(d)}}\sum_{r|f}
  \frac1{e_{r^2d_0}}h_{D,N}(r^2d_0),
\end{equation}
where
$$
\omega_{D_0}(d)=\#\{p\text{ prime}:~p|D_0,~p\nmid d\}
$$
and
$$
e_{r^2d_0}=\begin{cases}
  3, &\text{if }r^2d_0=-3, \\
  2, &\text{if }r^2d_0=-4, \\
  1, &\text{else}. \end{cases}
$$
(In accordance with the usual notation for Hurwitz class numbers,
$H_{D,N}$ are defined for nonnegative integers.)
Note that we have the formula
$$
h_{D,N}(r^2d_0)=h(r^2d_0)\prod_{p|DN}m(\O,r^2d_0,p),
$$
where
$$
m(\O,r^2d_0,p)=\begin{cases}
  \displaystyle 1-\JS{d_0}p, &\text{if }p|D\text{ and }p\nmid r, \\
  0, &\text{if }p|D\text{ and }p|r, \\
  \displaystyle 1+\JS{d_0}p, &\text{if }p|N\text{ and }p\nmid r, \\
  2, &\text{if }p|N\text{ and }p|r, \end{cases}
$$
is the number of local optimal embeddings of discriminant $r^2d_0$
into $\O_p:=\O\otimes_\Z\Z_p$. Define also that
\begin{equation} \label{equation: HDN(0)}
H_{D,N}(0)=\frac12\operatorname{Vol}(X_0^D(N))
=-\frac{D_0}{12}\prod_{p|D}\left(1-\frac1p\right)
\prod_{p|N}\left(1+\frac1p\right)
\end{equation}
(noticing that when $D=N=1$, $H_{1,1}(0)=-1/12=H(0)$),
and $H_{D,N}(x)=0$ if $-x$ is not $0$ or a negative discriminant.
With notations given as above, our main result can now be stated
as follows.

\begin{Theorem} \label{theorem: class number relation}
  Assume that $D_0$ is a positive squarefree integer.
  Let $Q(x,y)=ax^2+2bxy+cy^2$ be a positive definite
  quadratic form of discriminant $-16D_0$ such that
  every integer represented by $Q$ is congruent to $0$ or $1$ modulo
  $4$. Let $D$ and $N$ be defined by 
  \eqref{equation: DQ} and \eqref{equation: NQ}, respectively.
  Assume that $D>1$. Then for all positive integers $n$ congruent to
  $0$ or $1$ modulo $4$, we have the
  class number relations
  \begin{equation}\label{equation: new class number relation}
    \sum_{\substack{u,v\in\Z\\ u\equiv an\mod 2,~v\equiv cn\mod2 }}
    H_{D,N}\left(D_0n-\frac{Q(v,u)}{4}\right)=a_n
    H_{D,N}(0),
  \end{equation}
  where $a_n$ are the coefficients in
  $$
  \sum_{n=0}^\infty a_nq^n=\mathcal H(z):=
  \theta(z)^5-20\theta(z)\frac{\eta(4z)^8}{\eta(2z)^4}
  =1-10q-70q^4-48q^5+\cdots,
  $$
  and $H_{D,N}(m)$ and $H_{D,N}(0)$ are defined by \eqref{equation:
    HDN} and \eqref{equation: HDN(0)}, respectively.
\end{Theorem}

\begin{Remark} We remark that $\mathcal{H}(z)/120$ is Cohen's
  Eisenstein series of weight $5/2$ defined in
  \cite{Cohen-modularforms}.
\end{Remark}

Note that when $D_0$ is congruent to $3$ modulo $4$ and $Q$ is of the
form $4Q'$, the conditions $u\equiv an,v\equiv cn\mod 2$ imply that
$u$ and $v$ are both even. Thus, we may write the identity as
$$
\sum_{u,v\in\Z}H_{D,N}(D_0n-4Q'(u,v))=a_nH_{D,N}(0).
$$

We now explain the geometric meaning of these identities in more
details. Let $\O$ be an Eichler order of level $N$ in the
indefinite quaternion algebra $B$ of discriminant $D$ over $\Q$.
Let $\mathcal Q_{D,N}$ be the set of all points in $\mathscr A_2$
whose corresponding principally polarized abelian surfaces have
quaternionic multiplication by $\O$ such that the Rosati involutions
coincides with a positive involution of the form
$\alpha\mapsto\mu^{-1}\overline\alpha\mu$ on $\O$ for some $\mu\in\O$
with $\mu^2+DN=0$, where $\overline\alpha$ denotes the conjugate of
$\alpha$ in $\O$. It can be shown that $\mathcal Q_{D,N}$ consists of
a finite number $r_{D,N}$ of irreducible components, each of which is
the image of the Shimura curve $X_0^D(N)$ under some natural map (see
\cite{Rotger-Igusa} for the case $N=1$). In
\cite{Rotger-Crelle,Rotger-TAMS}, Rotger gave upper and lower bounds
for the number $r_{D,1}$ and in some cases an exact formula for
$r_{D,1}$. To give an exact formula for $r_{D,1}$, Lin
and Yang \cite{Yang-QM} showed that each irreducible component in
$\mathcal Q_{D,1}$ can be associated with a positive definite
quadratic form $Q$ of discriminant $-16D$ such that every positive
integer $a$ represented by $Q$ satisfies
\begin{enumerate}
  \item $a\equiv 0,1\mod 4$, and
  \item $\JS{-D,a}\Q\simeq B$.
\end{enumerate}
Furthermore, they showed that the irreducible components in $\mathcal
Q_{D,1}$ are in one-to-one correspondence with
$\mathrm{GL}(2,\Z)$-equivalence classes of such quadratic forms.
Hence the problem of determining $r_{D,1}$ reduces to that of
counting certain $\mathrm{GL}(2,\Z)$-equivalence classes of quadratic
forms. Note that the work of Rotger
\cite{Rotger-Crelle,Rotger-TAMS,Rotger-Igusa} and Lin and Yang
\cite{Yang-QM} can be easily extended to the case $N>1$. Moreover,
the condition $\JS{-D,a}\Q\simeq B$ is equivalent to that
$\chi_{p_j}(Q)=-1$ for $p_j|D$ and $\chi_{p_j}(Q)=1$ for $p_j|N$.
Thus, each quadratic form in Theorem \ref{theorem: class number
  relation} corresponds to a unique Shimura curve $\mathfrak X$ in
$\mathcal Q_{D,N}$. 

Now recall that the Humbert surface $H_n$ of discriminant $n$ is
defined to be the locus in $\A_2$ of principally polarized
abelian surfaces having a symmetric (Rosati invariant) endomorphism of
discriminant $n$. The Humbert surfaces are hypersurfaces in Siegel's
modular threefold. For each positive integer $n$ congruent to $0$ or
$1$ modulo $ 4$, we defined the Humbert divisor by
\begin{equation} \label{equation: Gn}
G_n:=\sum_{k^2\vert n}\nu_{n/k^2}H_{n/k^2},
\end{equation}
where the sum runs over any positive integer $k$ such that $n/k^2$ is an
integer congruent to $0$ or $1$ modulo $4$, $\nu_1=\nu_4=1/2,$ and
$\nu_{n/k^2}=1$ for other cases. 
Denote by $\bar{\A}_2$ the Satake compatification of $\A_2$.
A fundamental result of van der Geer
\cite[Theorem 8.1 and Corollary 8.2]{Geer-Siegel-threefold} showed that
the classes $G_n$ span a $1$-dimensional subspace of the $4$-th homology group of $\bar{\A}_2$ with coefficients in $\Q$, and
$G_n$ is the divisor of some Siegel modular form of weight
$-a_n/2$, where $a_n$ are the integers given in Theorem \ref{theorem:
  class number relation}. Thus, if we let $I(\X,G_n)$ denote the
intersection number of a Shimura curve $\X$ in $\mathcal Q_{D,N}$ with
$G_n$, then there exists a constant $c$ such that
$$
I(\X,G_n)=ca_n
$$
for all $n$. In other words, the generating series
$$
\sum_{n=0}^\infty I(\X,G_n)e^{2\pi inz}
$$
is a modular form of weight $5/2$ on $\Gamma_0(4)$, agreeing with
the general results of Kudla and Millson
\cite{Kudla-intersection-cycles}. (Here the constant $I(\X,G_0)$ of
the generating series should be interpreted as the volume of $\X$.)
Using the fact that the restriction of a Siegel modular form of weight
$k$ to $\X$ is a modular form of weight $2k$ on $X_0^D(N)$, we can
determine $c$ in terms of the volume of $X_0^D(N)$. 

On the other hand, the $0$-dimensional components of the intersection of $\X$ and
$G_n$, being moduli points corresponding to abelian surfaces with
endomorphism algebras strictly larger than the quaternion algebra of
discriminant $D$ over $\Q$, must be CM-points on $\X$. By carefully
analyzing which CM-points lie on the intersections and determining
their multiplicities, we obtain the class number relations in our main
theorem.

\begin{Remark} Note that the relations \eqref{equation:
    new class number relation} hold only for the case $D>1$. When
  $D=1$, the Shimura curve $X_0^1(N)$ is the usual modular curve
  and we also need to take the contribution from the cusps into
  account. In \cite[Satz 4]{Hermann}, Hermann obtained a formula for
  the contribution from the cusps in the case the Humbert surface is
  $H_1$ and the quadratic form is of the form $4Q'$ for some quadratic
  form $Q'$ of discriminant $-p$, where $p$ is a prime congruent to
  $3$ modulo $4$. His proof used the fact that the Siegel
  modular form with divisor $H_1$ is the product of the ten theta
  functions of even characteristics. However, because we do not have
  such knowledge about Siegel modular forms with divisor $H_n$ for
  general $n$ (except for some small $n$, see \cite{Gruenewald}),
  it is not easy to extend Hermann's formula to the case of general
  $n$ and $D_0$.\footnote{Fernando Villegas has some partial
    result in his private notes.}
\end{Remark}

\section{Shimura curves on Siegel's modular threefold}
In this section we shall introduce and review some properties of
Shimura curves on Siegel's modular threefold. The materials are taken
from \cite{Yang-QM,Rotger-Crelle,Rotger-TAMS,Rotger-Igusa}.

\subsection{Abelian surfaces with quaternionic multiplication}

Let $\O$ be an Eichler order of squarefree level $N$ in an indefinite
quaternion algebra $B$ of discriminant $D$ over $\Q$. Fix an
embedding
$$
\phi:B\hookrightarrow \mathrm{M}(2,\R).
$$
As a complex torus, an abelian surface with quaternionic multiplication
by $\O$ is necessarily isomorphic to $A_z=\C^2/\Lambda_z$ for some
$z\in\H$, where $\Lambda_z=\phi(\O)v_z$ with
$v_z=\left(\begin{smallmatrix} z\\ 1
\end{smallmatrix}\right).
$
Such a complex torus $A_z$ can always be principally
polarized. Namely, choose a pure quaternion $\mu$ in $\O$ such that
$\mu^2+DN=0$ and the $(2,1)$-entry of $\phi(\mu)$ is positive. We
define $E_\mu:\Lambda_z\x\Lambda_z\rightarrow\Z$ by
$$
E_\mu(\phi(\alpha)v_z,\phi(\beta)v_z)=\tr(\mu^{-1}\alpha\overline{\beta})
$$
and extend the definition of $E_\mu$ $\R$-bilinearly to $\C^2$. 
It can be shown that $E_\mu$ is a Riemann form on $\Lambda_z$ and the
corresponding polarization $\rho_\mu$ of $A_z$ is
principal. Conversely, if $A_z$ is a simple abelian surface, then all
principal polarizations of $A_z$ arise in this way.

Note that the polarization $\rho_\mu$ has the property that the
restriction of the Rosati involution with respect to $\rho_\mu$ to
$\O$ coincides with the involution
$\alpha\mapsto\mu^{-1}\overline\alpha\mu$, where $\overline\alpha$
denotes the quaternionic conjugate of $\alpha$. Thus, if we let
$\QQ_{D,N}$ denote the set of moduli points on $\A_2$ whose
corresponding principally polarized abelian surfaces have quaternionic
multiplication by $\O$ such that the Rosati involution restricted to
$\O$ coincides with $\alpha\mapsto\eta\overline\alpha\eta^{-1}$ for
some $\eta\in\O$ satisfying $\eta^2+DN=0$, then
$$
\QQ_{D,N}=\bigcup_\mu\X_\mu,
$$
where
$$
\X_\mu=\{(A_z,\rho_\mu):~z\in X_0^D(N)\},
$$
and $\mu$ are elements of $\O$ such that $\mu^2+DN=0$ and the
$(2,1)$-entry of $\phi(\mu)$ is positive. 

In \cite{Yang-QM}, in order
to determine the exact number of Shimura curves in $\QQ_{D,1}$, Lin
and Yang showed that there is a one-to-one correspondecne between
Shimura curves in $\QQ_{D,1}$ and 
$\mathrm{GL}(2,\Z)$-equivalence classes of positive definite binary
quadratic forms of discriminant $-16D$ such that all integers $a$
represented by the quadratc form satisfy $a\equiv 0,1\mod 4$ and
$\JS{-D,a}\Q\simeq B$. These quadratic forms arise
from singular relations satisfied by Shimura curve $\X_\mu$, which we
now review in the next section.

\subsection{Quadratic forms associated to Shimura curves}

In this section, we shall review the definition of quadratic forms
associated to Shimura curves in $\QQ_{D,N}$. These quadratic forms
appeared in the works of Hashimoto \cite{Hashimoto-modular-embedding}
and Runge \cite{Runge-endomorphism} and were studied in more details
in \cite{Yang-QM}.

Let $(A,\rho)$ be a principally polarized abelian surface over
$\C$. The polarization $\rho$ induces an involution $f\mapsto
f^\dagger$, called the Rosati involution, on the endomorphism ring of
$A$. In terms of the Riemann form $E$ associated to $\rho$, this
involution is characterized by the property that
$$
E(f\omega_1,\omega_2)=E(\omega_1,f^\dagger\omega_2)
$$
for all periods $\omega_1$ and $\omega_2$ of $A$. An endomorphism $f$
of $A$ is 
called symmetric or Rosati invariant if $f=f^\dagger.$ 
The symmetric endomorphisms can be described in terms of the period
matrix of $A$ as follows.

Let $\tau=\SM {\tau_1}{\tau_2}{\tau_2}{\tau_3}\in \H_2$ be a
normalized period matrix for $(A,\rho)$. An endomorphism $f$ of $A$
can be represented by a matrix $R_f\in \mathrm{M}(2,\C)$ such that 
\begin{equation}\label{Rf Mf}
R_f(\tau,1_2)=(\tau,1_2)M_f
\end{equation}
for some $M_f\in \mathrm M(4,\Z)$. Then $f$ is symmetric if and only
if $M_f$ satisfies
$$
M_f^t\M{0_2}{1_2}{-1_2}{0_2}=\M{0_2}{1_2}{-1_2}{0_2}M_f,
$$
i.e., if and only if $M_f$ is of the form
$$
M_f=\begin{pmatrix}
a_1&a_2&0&b\\
a_3&a_4&-b&0\\
0&c&a_1&a_3\\
-c&0&a_2&a_4
\end{pmatrix}
$$
for some integers $a_j, b$ and $c$. Then \eqref{Rf Mf} implies that
$$
a_2\tau_1+(a_4-a_1)\tau_2-a_3\tau_3+b(\tau^2_2-\tau_1\tau_3)+c=0.
$$
Note that since $f\in\Z$ if and only if $a_2=a_3=b=c=0$ and $a_1=a_4$,
the relation is nontrivial if and only if $f\notin\Z$.

Conversely, if $\tau\in\H_2$ satisfies
\begin{equation}\label{equation: singular relation}
a\tau_1+b\tau_2+c\tau_3+d(\tau^2_2-\tau_1\tau_3)+e=0,\ \ \ a, b, c, d,e\in \Z,
\end{equation}
then its associated abelian surface has
an endomorphism $f$ with
\begin{equation}\label{matrix M}
M_f=\begin{pmatrix}
0&a&0&d\\
-c&b&-d&0\\
0&e&0&-c\\
-e&0&a&b
\end{pmatrix}.
\end{equation}

A relation of this type (\ref{equation: singular relation}) for
$\ell=(a,b,c,d,e)\in\Z^5$ is called a singular relation satisfied by
$\tau$ (or by its associated abelian surface $(A,\rho)$). A singular
relation is called primitive if 
$\mathrm{gcd}(a,b,c,d,e)=1$. Note that the matrix $M_f$ above
satisfies 
$
x^2-bx+ac+de=0.
$
Thus, the discriminant for a singular relation  $\ell=(a,b,c,d,e)$ is defined naturally as the discriminant of the polynomial, i.e., as 
$$\Delta(\ell)=b^2-4(ac+de).$$
Then the property that $\Im\tau$ is positive definite implies that
$\Delta(\ell)>0$. We also remark that
Humbert proved that under the natural projection
$\H_2\to\A_2$, all primitive singular relations with the same
discriminant $n$ of the form \eqref{equation: singular relation}
defines the same zero locus in $\A_2$, which is called the Humbert surface
$H_n$ of discriminant $n$.

Observe that for a given $\tau\in\H_2$, if we let
$\mathcal{L}$ be the set of singular 
relations satisfied by $\tau$, then 
$\Delta$ defines a positive definite quadratic form on $\mathcal L$
so that $\mathcal L$ becomes a positive definite lattice.
Let $\gen{\cdot,\cdot}_\Delta$ denote the bilinear form associated to $\Delta$, i.e.,
\begin{equation}\label{equation: inner P}
\gen{\ell_1,\ell_2}_\Delta=\frac{1}{2}(\Delta(\ell_1)+\Delta(\ell_2)-\Delta(\ell_1+\ell_2)).
\end{equation}
By a direct computation, we find that
\begin{equation}\label{congruence relation}
\gen{\ell_1,\ell_2}^2_\Delta\equiv\Delta(\ell_1)\Delta(\ell_2)\mod 4
\end{equation}
for all singular relations $\ell_1$, $\ell_2$ satisfied by a given
$\tau\in\H_2$.

We now consider the case $(A,\rho)=(A_z,\rho_\mu)\in\X_\mu$,
where $\X_\mu$ is a Shimura curve in $\QQ_{D,N}$. Since the Rosati
involution restricted to the Eichler order $\O$ coincides with the
involution
$
\alpha\mapsto\mu^{-1}\bar{\alpha}\mu,
$
an element $\alpha$ of $\O$ is Rosati invariant if and only if
$$
\alpha\in\mu^\perp=\braces{\beta\in\O:\tr(\beta\bar{\mu})=0}.
$$
Therefore, for all $(A,\rho)$ in $\X_\mu$, the lattice of singular
relations satisfied by $(A,\rho)$ contains a sublattice
isomorphic to $\mu^\perp/\Z$, and there is a natural quadratic form on
$\mu^\perp/\Z$ defined by
$\alpha+\Z\mapsto\operatorname{disc}(\alpha)
=\tr(\alpha)^2-4\nm(\alpha)$. Consequently, we may associate to
$\X_\mu$ a $\GL(2,\Z)$-equivalence class of quadratic forms
$Q_\mu(x,y)$ defined by
$$
Q_\mu(x,y)=\operatorname{disc}(\alpha x+\beta y)
=\tr(\alpha x+\beta y)^2-4\nm(\alpha x+\beta y)
$$
where $\{\alpha,\beta\}$ is a $\Z$-basis for $\mu^\perp/\Z$.

\begin{theorem}[{\cite[Theorem 3 and Proposition 27]{Yang-QM}}]
  \label{theorem: correspondence}
  \begin{enumerate}
  \item
    The quadratic form $Q_\mu$ above has the property that all
    positive integers $a$ represented by $Q_\mu$ satisfy
    \begin{enumerate}
    \item
      $a\equiv 0,1\mod 4$, and
    \item that the quaternion algebra $\JS{-DN,a}\Q$ has discriminant
      $D$.
    \end{enumerate}
  \item Conversely, assume that $N$ is squarefree. Then each positive
    definite binary quadratic forms $Q$ of discriminant $-16DN$ having
    the two properties in Part (1) is $\GL(2,\Z)$-equivalent to
    $Q_\mu$ for some $\mu\in\O$.
  \item Assume that $N$ is squarefree. Then the correspondence
    $\X_\mu\longleftrightarrow Q_\mu$ between Shimura curves in
    $\QQ_{D,N}$ and $\GL(2,\Z)$-equivalence
    classes of quadratic forms of discriminant $-16DN$ having the
    properties in Part (1) is one-to-one and onto.
  \item If $Q_\mu(x,y)$ is not ambiguous, i.e., if $Q_\mu(-x,y)$ is
    not $\SL(2,\Z)$-equivalent to $Q_\mu(x,y)$, then $\X_\mu\simeq
    X_0^D(N)/\gen{w_{DN}}$. If $Q_\mu$ is ambiguous, then $Q_\mu$
    primitively represents $m$ or $4m$ for some positive divisor $m$ of $DN$
    and we have $\X_\mu\simeq X_0^D(N)/\gen{w_m,w_{DN}}$. Here $w_n$
    denotes the Atkin-Lehner operator on $X_0^D(N)$ induced by an
    element of norm $n$ in $\O$.
  \end{enumerate}
\end{theorem}

Note that in \cite{Yang-QM}, we only proved the case $N=1$, but it is
easy to see that the proof also works
for the case of squarefree $N$ with $(D,N)=1$.

\begin{Remark}
  Note that if $D_0$ is not squarefree, then there may exist a
  quadratic form of discriminant $-16D_0$ that is neither primitive
  nor of the form $4Q'$ for some primitive $Q'$, but all integers
  represented by it are congruent to $1$ or $4$ modulo $4$. In such a
  case, there may exist different and nonisomorphic Shimura
  curves in $\A_2$ with this quadratic form. See \cite[Example
  13]{Runge-endomorphism} for an example.
\end{Remark}


\section{Singular relations satisfied by points on Shimura curves}

As before, we assume that $D_0$ is a positive
squarefree integer. For a positive definite binary quadratic form $Q$
of discriminant $-16D_0$ such that all integers represented by $Q$
are congruent to $0$ or $1$ modulo $4$, define $D$ and $N$ by
\eqref{equation: DQ} and \eqref{equation: NQ}, respectively.
Let $\X$ be the Shimura curve in $\QQ_{D,N}$ associated to $Q$ as
described in Theorem \ref{theorem: correspondence}.
To facilitate our discussion on the intersection of $\X$
and Humbert surfaces, in this section, we shall make
singular relations satisfied by points on $\X$ explicitly.
We recall that $Q$ is either primitive or is $4Q'$ for some quadratic
form $Q'$ of discriminant $-D_0$.

\subsection{Singular relations satisfied by Shimura curves}
\begin{Lemma}[{\cite[Lemma 38]{Yang-QM}}]
  \label{lemma: explicit mu perp}
  \begin{enumerate}
  \item
    Assume that $Q$ is primitive. Let $p$ be a prime represented by
    $Q$ such that $p\nmid DN$ (so that $p\equiv 1\mod 4$ and the
    quaternion algebra $\JS{-DN,p}\Q$ has discriminant $D$). Let
    $B=\JS{-DN,p}\Q$ be the quaternion algebra generated by $I$ and
    $J$ with $I^2=-DN$, $J^2=p$, and $IJ=-JI$. Choose an even integer
    $s$ such that $s^2DN+1\equiv 0\mod p$. Then the $\Z$-module $\O$
    spanned by
    \begin{equation}\label{equation: Eichler basis 1}
    e_1=1, \quad e_2=\frac{1+J}2, \quad
    e_3=\frac{I+IJ}2, \quad e_4=\frac{sDNJ+IJ}p,
    \end{equation}
    is an Eichler order of level $N$ in $B$. Moreover, with the choice
    $\mu=I$, the $\Z$-module $\mu^\perp/\Z$ is spanned by $e_2$ and
    $e_4$ and hence the quadratic form $Q_\mu$ defined by
    $Q_\mu(x,y)=\disc(e_2x+e_4y)$ is $$px^2+4sDNxy+4tDNy^2$$ (which is
    necessarily $\GL(2,\Z)$-equivalent to $Q$ since $p$ is a prime),
    where $t=(s^2DN+1)/p$.
  \item
    Assume that $Q=4Q'$ for some quadratic form $Q'$ of discriminant
    $-DN$. Let $p$ be a prime represented by $Q'$ and
    $B=\JS{-DN,p}\Q$. Let $s$ be an odd integer such that
    $s^2DN+1\equiv 0\mod 4p$. Then the $\Z$-module $\O$ spanned by
    \begin{equation}\label{equation: Eichler basis 2}
    e_1=1, \quad e_2=\frac{1+I}2, \quad
    e_3=J, \quad e_4=\frac{sDNJ+IJ}{2p}
    \end{equation}
    is an Eichler order of level $N$ in $B$. Moreover, with the choice
    $\mu=I$, the $\Z$-module $\mu^\perp/\Z$ is spanned by $e_3$ and
    $e_4$ and hence the quadratic form $Q_\mu$ is
    $$4px^2+4sDNxy+4tDNy^2$$ (which is $\GL(2,\Z)$-equivalent to $Q$),
    where $t=(s^2DN+1)/4p$.
  \end{enumerate}
\end{Lemma}

Note again that Lemma 38 of \cite{Yang-QM} considered only the case
$N=1$, but it is easy to see that it can be generalized as above.
Note also that the construction of Eichler orders of the form given in
the lemma first appeared in \cite{Ibukiyama-orders} and was used in
\cite{Hashimoto-modular-embedding}.

Let $B=\JS{-DN,p}\Q$ be the quaternion algebra of discriminant $D$
over $\Q$ and $\O$ be the Eichler order of level $N$ given in the
lemma. Choose the embedding $\phi:B\hookrightarrow \mathrm{M}(2,\R)$ to be
$$
\phi(I)=\M0{-1}{DN}0, \quad \phi(J)=\M{\sqrt p}00{-\sqrt p}.
$$
Let $\alpha_1,\ldots,\alpha_4$ be a symplectic basis with respect to
the element $\mu=I$ in the lemma, i.e.,
$$
\left(\tr(\mu^{-1}\alpha_i\overline\alpha_j)\right)_{i,j=1,\ldots,4}
=\M{0_2}{1_2}{-1_2}{0_2}.
$$
For $z\in\H$, let
$v_z=\left(\begin{smallmatrix}z\\1\end{smallmatrix}\right)$,
$\Lambda_z=\phi(\O)v_z$, and $A_z=\C^2/\Lambda_z$. Write
$$
\left(\phi(\alpha_1)v_z,\ldots,\phi(\alpha_4)v_z\right)
=(\tau_1,\tau_2),
$$
where $\tau_j\in \mathrm{M}(2,\C)$. Then the normalized period matrix for the
principally polarized abelian surface $(A_z,\rho_\mu)$ is
$\tau_z:=\tau_2^{-1}\tau_1$. The quaternion modular embedding
$$
\iota_\mu:\H\to\H_2, \quad \iota_\mu:z\mapsto\tau_z
$$
defines an explicit map from $X_0^D(N)$ to $\X_\mu$ in
$\QQ_{D,N}$. Moreover, for $\gamma\in\O$, the matrix for $\gamma$, as
an endomorphism of $A_z$, with respect to the basis
$\phi(\alpha_1)v_z,\ldots,\phi(\alpha_4)v_z$ is the matrix
$M_\gamma$ such that
$$
(\gamma\alpha_1,\ldots,\gamma\alpha_4)
=(\alpha_1,\ldots,\alpha_4)M_\gamma.
$$
From this, we obtain the following description of singular relations
satisfied by the Shimura curve associated to $Q$.

\begin{Lemma}
  \label{lemma: singular relations for X}
  We retain the notations above and those in Lemma \ref{lemma:
    explicit mu perp}.
  \begin{enumerate}
  \item Assume that $Q$ is primitive. A symplectic basis for $\O$ with
    respect to $\mu=I$ is
    $$
    \alpha_1=e_3-\frac{p-1}{2}e_4,\quad \alpha_2=-sDNe_1-e_4,\quad
    \alpha_3=e_1,\quad \alpha_4=e_2.
    $$
    The normalized period matrix of $(A_z,\rho_\mu)$ with respect to
    this symplectic basis is
    \begin{equation}\label{equation: Omega_z 1}
      \tau_z=\frac{1}{pz}\M
      {-\bar{\epsilon}^2+(p-1)sDNz/2+DN\epsilon^2z^2}
      {\bar{\epsilon}-(p-1)sDNz-DN\epsilon z^2}
      {\bar{\epsilon}-(p-1)sDNz-DN\epsilon z^2}
      {-1-2sDNz+DNz^2},
    \end{equation}
    where $\epsilon=(1+\sqrt{p})/2$ and
    $\overline{\epsilon}=(1-\sqrt{p})/2$. Moreover, the singular
    relations corresponding to $e_2$ and $e_4$ are
    \begin{equation}\label{equation: Q mu basis 1}
      \ell_1=(1,1,(1-p)/4,0,0),\quad\ell_2=(0,2sDN,0,1,DN(s^2DN-t)),
    \end{equation}
    respectively.
  \item Assume that $Q=4Q'$ for some quadratic form $Q'$ of
    discriminant $-DN$. A symplectic basis for $\O$ with respect to
    $\mu=I$ is
    $$
    \alpha_1=e_2,\quad\alpha_2=-e_4,\quad
    \alpha_3=e_1,\quad \alpha_4=e_3.
    $$
    The normalized period matrix of $(A_z,\rho_\mu)$ with respect to
    this symplectic basis is
    \begin{equation}\label{equation: Omega_z 2}
      \tau_z=\frac{1}{4z}\M
      {DNz^2+2z-1}{-DNz^2/\sqrt{p}-1/\sqrt{p}}
      {-DNz^2/\sqrt{p}-1/\sqrt{p}}{DNz^2/p-2sDNz/p-1/p}.
    \end{equation}
    Moreover, the singular relations corresponding to $e_3$ and $e_4$
    are
    \begin{equation}\label{equation: Q mu basis 2}
      \ell_1=(1,0,-p,0,-(1+sDN)/2),\quad
      \ell_2=(0,0,(1-sDN)/2,1,-tDN),
    \end{equation}
    respectively.
  \end{enumerate}
\end{Lemma}

Essentially, everything (in the case $N=1$) in the lemma is contained
in the proof of Theorem 5.1 of \cite{Hashimoto-modular-embedding} and Lemma 40 of \cite{Yang-QM}, although in the proofs therein the
singular relations in \eqref{equation: Q mu basis 1} and
\eqref{equation: Q mu basis 2} were obtained by direct computation
using the description of $\tau_z$ above. Here we point out that these
singular relations correspond to the basis for $\mu^\perp/\Z$ given in
Lemma \ref{lemma: explicit mu perp}.

Note that the explicit description of $\ell_1$ and $\ell_2$ in
\eqref{equation: Q mu basis 1} and \eqref{equation: Q mu basis 2}
shows that for $m_1,m_2\in\Z$, the singular relation
$m_1\ell_1+m_2\ell_2$ is primitive if and only if $(m_1,m_2)=1$.
Consequently, we have the following result.

\begin{Corollary} A Shimura curve $\X$ in $\QQ_{D,N}$ is contained in
  the Humbert surface of discriminant $n$ if and only if its
  associated quadratic form primitively represents $n$.
\end{Corollary}

Note that if $z$ is not a CM point, then the lattice of singular
relation satisfied by $\tau_z$ has rank $2$ and is spanned by $\ell_1$
and $\ell_2$ in the lemma. In the next section, we
will consider the case of CM-points.

\subsection{Singular relations satisfied by CM-points}

We retain all the notations in the previous section. In this section,
we shall describe singular relations satisfied by CM-points on a
Shimura curve $\X$ associated to a quadratic form $Q$ of discriminant
$-16DN$.

Recall that a CM-point $z$ on $X_0^D(N)$ is defined to be the common
fixed point of $\phi(\psi(K))$ in the upper half-plane for some
embedding $\psi$ from an imaginary quadratic field $K$ into $B$. Its
discriminant is defined to be the discriminant of the quadratic order $R$ in
$K$ such that $\psi(K)\cap\O=\psi(R)$. To describe singular relations
satisfied by the corresponding point $\tau_z$ on $\X$, we let
\begin{equation}\label{equation: basis of trace zero}
  \beta_1=2e_2-e_1=J,\quad\beta_2=e_3=\frac{I+IJ}2,\quad
  \beta_3=e_4=\frac{sDNJ+IJ}p
\end{equation}
in the case $Q$ is primitive, and
\begin{equation}\label{equation: basis for trace zero 2}
  \beta_1=e_3=J, \quad \beta_2=2e_2-e_1=I, \quad
  \beta_3=e_4=\frac{sDNJ+IJ}{2p}
\end{equation}
in the case $Q=4Q'$. Then they form a basis for the $\Z$-module of elements
of trace zero in $\O$.

\begin{Lemma}[{\cite[Lemma 40]{Yang-QM}}]
  \label{lemma: discriminant of CM lattice}
  Let $z$ be a CM-point of discriminant $d$ on $X_0^D(N)$, $\psi$ be
  its corresponding optimal embedding of discriminant $d$,
  $\tau_z$ be its corresponding point on $\X$, and $\mathcal L_z$ be
  the lattice of singular relations satisfied by $\tau_z$. Write
  $$
  \psi(\sqrt d)=b_1\beta_1+b_2\beta_2+b_3\beta_3.
  $$
  \begin{enumerate}
  \item Assume that $Q$ is primitive. Let $\ell_1$ and $\ell_2$ be the
    singular relations satisfied by $\X$ given in \eqref{equation: Q
      mu basis 1}. Then in addition to $\ell_1$ and $\ell_2$, $\tau_z$
    also satisfies
    $$
    \ell_3=(0,b_2,b_3+b_2(1-p)/4,0,b_1).
    $$
    These three singular relations form a $\Z$-basis for $\mathcal
    L_z$. Moreover, the Gram matrix of $\mathcal L_z$ with respect to
    this basis is
    \begin{equation}\label{equation: CM Gram1}
    \begin{pmatrix}
    p&2sDN&-b_2p/2-b_3\\
    2sDN&4tDN&2b_1-b_2sDN\\
    -b_2p/2-b_3&2b_1-b_2sDN&b^2_2/4
    \end{pmatrix}.
    \end{equation}
    (Note that $b_2$ and $b_3$ are necessarily even, while $b_1$ is
    even or odd depending on whether $d$ is even or odd.)
  \item Assume that $Q=4Q'$ for some quadratic form $Q'$ of discriminant
    $-DN$. Let $\ell_1$ and $\ell_2$ be the
    singular relations satisfied by $\X$ given in \eqref{equation: Q
      mu basis 2}. Then in addition to $\ell_1$ and $\ell_2$, $\tau_z$
    also satisfies
    $$
    \ell_3=(0,-b_2,b_3/2,0,-b_1/2).
    $$
    These three singular relations form a $\Z$-basis for $\mathcal
    L_z$. Moreover, the Gram matrix of $\mathcal L_z$ with respect to
    this basis is
     \begin{equation}\label{equation: CM Gram2}
       \begin{pmatrix}
       4p&2sDN&-b_3\\
       2sDN&4tDN&b_1\\
       -b_3&b_1&b_2^2
       \end{pmatrix}.
     \end{equation}
     (Note that $b_1$ and $b_3$ are necessarily even, while $b_2$ is
     even or odd depending on whether $d$ is even or odd.)
  \end{enumerate}
  In both cases, the discriminant of $\mathcal L_z$ is $4|d|$.
\end{Lemma}

\begin{proof}
  The case where $Q$ is primitive was proved (in the case
  $N=1$) in Lemma 40 of \cite{Yang-QM}. Here we sketch a proof for the
  case $Q=4Q'$.

  Recall from Lemma \ref{lemma: singular relations for X} that
  $\tau_z$ is given by
  $$
  \tau_z=\frac{1}{4z}\M
      {DNz^2+2z-1}{-DNz^2/\sqrt{p}-1/\sqrt{p}}
      {-DNz^2/\sqrt{p}-1/\sqrt{p}}{DNz^2/p-2sDNz/p-1/p}.
  $$
  Write $\tau_z$ as $\SM{\tau_1}{\tau_2}{\tau_2}{\tau_3}$. We have
  $$
  \begin{pmatrix}\tau_1\\\tau_2\\\tau_3\\\tau_2^2-\tau_1\tau_3\\1
  \end{pmatrix}
  =\frac14\begin{pmatrix}
    1&2&1\\1/\sqrt p&0&-1/\sqrt p\\
    1/p&-2sDN/p&1/p\\
    (sDN-1)/2p&DN(1+s)/p&(sDN-1)/2p\\0&4&0\end{pmatrix}
  \begin{pmatrix}-1/z\\1\\DNz\end{pmatrix}.
  $$
  Write $\psi(\sqrt d)$ as $\psi(\sqrt d)=c_1I+c_2J+c_3IJ$. Changing
  $\psi(\sqrt d)$ to $-\psi(\sqrt d)$ if necessary, we may assume that
  $c_1+c_3\sqrt p>0$. Since $z$ is fixed by $\psi(\sqrt{d})$, we have
  $$
  z=\frac{c_2\sqrt p+\sqrt d}{DN(c_1+c_3\sqrt p)},\qquad
  -\frac1z=\frac{-c_2\sqrt p+\sqrt d}{c_1-c_3\sqrt p},
  $$
  and
  $$
  \begin{pmatrix}-1/z\\1\\DNz\end{pmatrix}
  =\begin{pmatrix}\gamma&\delta\\0&1\\\overline\gamma&\overline\delta
  \end{pmatrix}\begin{pmatrix}\sqrt d\\1\end{pmatrix},
  \qquad
  \gamma=\frac1{c_1-c_3\sqrt p}, \qquad
  \delta=\frac{-c_2\sqrt p}{c_1-c_3\sqrt p},
  $$
  where $\overline\gamma$ and $\overline\delta$ are Galois conjugates
  of $\gamma$ and $\delta$, respectively. Hence, $(a_1,\ldots,a_5)$ is
  a singular relation for $\tau_z$ if and only if $(a_1,\ldots,a_5)$
  is in the nullspace of
  $$
  \begin{pmatrix}
    1&2&1\\1/\sqrt p&0&-1/\sqrt p\\
    1/p&-2sDN/p&1/p\\
    (sDN-1)/2p&DN(1+s)/p&(sDN-1)/2p\\0&4&0\end{pmatrix}
  \begin{pmatrix}\gamma&\delta\\0&1\\\overline\gamma&\overline\delta
  \end{pmatrix}.
  $$
  Using the relations
  $$
  c_1=b_2, \qquad c_2=b_1+\frac{b_3sDN}{2p}, \qquad
  c_3=\frac{b_3}{2p},
  $$
  we check directly that $\ell_3=(0,-b_2,b_3/2,0,-b_1/2)$ is a
  singular relation satisfied by $\tau_z$.
  
   Finally, we observe that
  the assumption that $\psi$ is an optimal embedding of discriminant
  $d$ implies that
  $$
  \gcd(b_1,b_2,b_3)=\begin{cases}
    1, &\text{if }d\text{ is odd},\\
    2, &\text{if }d\text{ is even}, \end{cases}
  $$
  which implies that if $\ell$ is a singular relation satisfied by
  $\tau_z$ and $r_1,r_2,r_3$ are the coefficients in
  $\ell=r_1\ell_1+r_2\ell_2+r_3\ell_3$, then $r_1,r_2,r_3$ are all
  integers. We conclude that
  $\ell_1$, $\ell_2$, $\ell_3$ form a $\Z$-basis for $\mathcal L_z$.
\end{proof}

\subsection{Galois orbits of CM-points}

Observe that if $\psi:K\hookrightarrow B$ is an embedding from an
imaginary quadratic field $K$ into $B$, then so is
$\overline\psi:a\mapsto\overline{\psi(a)}$, where $\overline{\psi(a)}$
denotes the quaternionic conjugate of $\psi(a)$, and $\phi(\psi(K))$
and $\phi(\overline\psi(K))$ have a common fixed point $z_\psi$ in the
upper half-plane. To remove the ambiguity when
we talk about the embedding corresponding to $z_\psi$, we say
an embedding $\psi$ is normalized (with respect to $\phi$) if
$$
\phi(\psi(a))\begin{pmatrix}
z_\psi\\ 1
\end{pmatrix}=a\begin{pmatrix} z_\psi\\1\end{pmatrix}
$$
for all $a\in K$ (as opposed to $\phi(\psi(a))
\left(\begin{smallmatrix}z_\psi\\1\end{smallmatrix}\right)
=\overline a\left(\begin{smallmatrix}z_\psi\\1\end{smallmatrix}
  \right)$, where $\overline a$ is the complex conjugate of $a$).
It is clear that $\psi$ is normalized with respect to $\phi$ if and
only if the $(2,1)$-entry of $\phi(\psi(a))$ has the same sign as the
imaginary part of $a$ for any $a\in K$. From now on, whenever we
mention the embedding corresponding to a CM-point, we refer to the one
that is normalized.

Denote by $\CM(d)$ the set of CM-points of
discriminant $d$ on $X_0^D(N)$. By local consideration, we have
\begin{equation}\label{number of CM-points of discriminant d}
\abs{\CM(d)}=h(d)\prod_{q\text{ primes}}m(\O,d,q),
\end{equation}
where $h(d)$ is the class number of $R$ and $m(\O,d,q)$ is the number
of local optimal embeddings of $R$ into $\O\otimes\Z_q$. The formula
for $m(\O,d,q)$ is given by
\begin{equation}\label{equation: number of local optimal embeddings}
m(\O,d,q)=\begin{cases}1,\ &\text{if }q\nmid DN,\\
1-\left(\frac{d}{q}\right),&\text{if }q\vert D\text{ and }\nmid f,\\
0,&\text{if }q\vert D\ \text{ and }q\vert f,\\
1+\left(\frac{d}{q}\right),&\text{if }q\vert N\text{ and }q\nmid f,\\
2,&\text{if }q\vert N\text{ and }p\vert f,\end{cases}
\end{equation}
assuming $N$ is squarefree, where $f$ is the positive integer such
that $d=f^2d_0$ for a fundamental discriminant $d_0$.

Let $L$ be the ring class field of $R$. Then the class group of $R$
acts on $\CM(d)$ and two CM-points of discriminant $d$ are in the same
orbit under this action if and only if their corresponding
(normalized) optimal embeddings are locally equivalent at every
place. In view of the Shimura reciprocity law, this amounts to the
property that two CM-points of discriminant $d$ are in the same
$\Gal(L/K)$-orbit if and only if their corresponding optimal
embeddings are locally equivalent at every place. This leads to the
following explicit criterion when two CM-points are
$\Gal(L/K)$-conjugates.

\begin{Lemma} \label{lemma: Galois orbits}
  Let $z$ and $z'$ be two CM-points on $X_0^D(N)$ and
  $\psi$ and $\psi'$ be their respective normalized optimal embeddings
  into $\O$. Write $\psi(\sqrt d)=b_1\beta_1+b_2\beta_2+b_3\beta_3$
  and $\psi'(\sqrt d)=b_1'\beta_1+b_2'\beta_2+b_3'\beta_3$. Write
  $d=f^2d_0$ for a fundamental discriminant $d_0$. We have the
  following properties.
  \begin{enumerate}
  \item
    We have
    $$
    b_1^2p\equiv d\mod\begin{cases}
    4DN &\text{in the case }Q\text{ is primitive},\\
    DN  &\text{in the case }Q=4Q'.\end{cases}
    $$
  \item Let $q$ be a prime divisor of $DN$. If $\psi$ and $\psi'$ are
    equivalent locally at $q$, then
    $$
    b_1\equiv b_1'\mod\begin{cases}
      4, &\text{if }q=2, \\
      q, &\text{if }q\neq 2.\end{cases}
    $$
    Moreover, if $q$ does not divide $(f,N)$, then this is necessary
    and sufficient condition for $\psi$ and $\psi'$ to be equivalent
    at $q$.
  \item If $z$ and $z'$ lie in the same $\Gal(L/K)$-orbit, then
  $$
  b_1\equiv b_1'\mod \begin{cases}
    2DN &\text{in the case }Q\text{ is primitive}, \\
    DN &\text{in the case }Q=4Q'.\end{cases}
  $$
  Moreover, if $(N,f)=1$, then this is a necessary and sufficient
  condition for $z$ and $z'$ to be in the same $\Gal(L/K)$-orbit.
  In such a case, the set of $\Gal(L/K)$-orbits of CM-points of
  discriminant $d$ on $X_0^D(N)$ are in one-to-one correspondence with
  the set
  $$
  \{r\mod 2DN:~pr^2\equiv d\mod 4DN\}
  $$
  in the case $Q$ is primitive, and with the set
  $$
  \{r\mod DN:~pr^2\equiv d\mod DN\}
  $$
  in the case $Q=4Q'$.
  \end{enumerate}
\end{Lemma}

\begin{proof} The case of primitive $Q$ was proved in details in Lemma
  46 of \cite{Yang-QM}. Here we sketch the proof of the case $Q=4Q'$.

  We have
  $$
  d=-\nm(\psi(\sqrt{d}))=b_1^2p+DN(
  b_1b_3s-b_2^2+b_3^2t),\qquad t=\frac{s^2DN+1}{4p}.
  $$
  Therefore, we have $b_1^2p\equiv d\mod DN$.

  Now let $q$ be a prime divisor of $DN$. Since $q$ is odd and
  different from $p$, $1$, $I$, $J$, and $IJ$ form a $\Z_q$-basis for
  $\O_q:=\O\otimes_\Z\Z_q$. A direct computation shows that if
  $\gamma=d_1+d_2I+d_3J+d_4IJ\in\O_q$ and
  $c'I+c_2'J+c_3'IJ=\gamma(c_1I+c_2J+c_3IJ)\gamma^{-1}$, then
  $$
  c_2'-c_2=\frac{2DN(c_1d_0d_3-c_3d_0d_1+c_1d_1d_2-c_2d_1^2
    -c_3d_2d_3p+c_2d_3^2p)}{\nm(\gamma)}.
  $$
  Therefore, if $\gamma\in\O_q^\times$, then $c_2'\equiv c_2\mod q$,
  which implies that if $\psi$ and $\psi'$ are locally equivalent at
  $q$, then $b_1'\equiv b_1\mod q$. In the case $q\nmid(N,f)$, the
  number of solutions of the congruence equation $pr^2\equiv d\mod q$
  coincides with the number of local optimal embeddings discriminant
  $d$ at $q$. Thus, in the case $q\nmid(N,f)$, local optimal
  embeddings of discriminant $d$ at the place $q$ are in one-to-one
  correspondence with the solutions of $pr^2\equiv d\mod q$. Then the
  lemma follows.
\end{proof}

\begin{Remark}
Note that when $q|(N,f)$, the number of local optimal embeddings of
discriminant $d$ is $2$ by \eqref{equation: number of local optimal
  embeddings}, but the congrunce equation $b_1^2p\equiv d\mod q$ has
only one solution, so the condition for local equivalence of $\psi$
and $\psi'$ at $q$ given in the lemma does not work. It
is possible to work out a condition in such a case, but it is not
needed for the purpose of this paper.
\end{Remark}
 
\begin{Corollary}\label{e of singular relations}
  Let $\X$ be a Shimura curve in $\QQ_{D,N}$.
  Let $d$ be a negative discriminant such that CM-points of
 discriminant $d$ exist on the Shimura curve $\X$. Let
 $K=\Q(\sqrt{d})$, $R$ be the quadratic order of discriminant $d$, and
 $L$ be the ring class field of $R$.
 If $z$ and $z'$ are two CM-points of discriminant $d$ lying in the
 same $\mathrm{Gal}(L/K)$-orbit, then their lattices of singular
 relations $\mathcal L_z$ and $\mathcal L_{z'}$ are isomorphic.
\end{Corollary}

\begin{proof}
  Let $Q$ be the quadratic form associated to $\X$. The case $Q$ is
  primitive was proved in \cite[Proposition 41]{Yang-QM} (for the case
  $N=1$). Here we provide a proof for the case $Q=4Q'$ (so that $DN$
  is assumed to be congruent to $3$ modulo $4$).
  
  Let $B$, $\O$, and $\mu$ be given as in Part (2) of Lemma
  \ref{lemma: explicit mu perp}, and $\beta_1,\beta_2,\beta_3$ be
  given by \eqref{equation: basis for trace zero 2}.
  Assume that $z$ and $z'$ are CM-points of discriminant $d$ and
  $\psi$ and $\psi'$ are their respective normalized optimal
  embeddings. Write $\psi(\sqrt d)=b_1\beta_1+b_2\beta_2+b_3\beta_3$
  and $\psi'(\sqrt d)=b'_1\beta_1+b'_2\beta_2+b'_3\beta_3$. Recall
  from \eqref{equation: CM Gram2} that the Gram matrices for $z$ and
  $z'$ are
  \begin{equation*}
  A=\begin{pmatrix}
  4p&2sDN&-b_3\\
  2sDN& 4tDN&b_1\\
  -b_3&b_1&b_2^2
  \end{pmatrix},\ \ \ A'= \begin{pmatrix}
  4p&2sDN&-b_3'\\
  2sDN& 4tDN&b_1'\\
  -b_3'&b_1'&(b_2')^2
  \end{pmatrix}.
  \end{equation*}
  It is clear that $A'=UAU^{t}$, where
  $$
  U=\begin{pmatrix}
  1&0&0\\
  0&1&0\\
  u&v&1
  \end{pmatrix}
  $$
  with
  $$
  u=t(b_3-b_3')+\frac{s}{2}(b_1-b_1'),\quad
  v=-\frac{s}{2}(b_3-b_3')-\frac{p}{DN}(b_1-b_1').
  $$
  Since $\psi(\sqrt d)/2\in\O$ or $(1+\psi(\sqrt d))/2\in\O$
  depending on whether $d$ is even or odd, the integers $b_1$ and
  $b_3$ are necessarily even, and so are $b_1'$ and $b_3'$.
  Furthermore, by Lemma \ref{lemma: Galois orbits}, one has $b_1\equiv
  b_1'\mod DN$. Consequently, the numbers $u$ and $v$ are both
  integers and $\mathcal L_z\simeq\mathcal L_{z'}$. This proves the
  assertion. 
 \end{proof}

\section{Intersection of Shumura curves with Humbert surfaces}

As in the previous sections, we assume that $D_0$ is a positive
squarefree integer. Let $Q$ be a quadratic form of discriminant
$-16D_0$ appearing in the statement of Theorem \ref{theorem: class
  number relation}, $D$ and $N$ be defined by \eqref{equation: DQ} and
\eqref{equation: NQ}, respectively, and $\X$ be the Shimura curve
corresponding to $Q$ in $\QQ_{D,N}$ as described in Theorem
\ref{theorem: correspondence}. 

In this section, we study the
intersection of the Shimura curve $\X$ with Humbert divisors
$G_n$ defined by \eqref{equation: Gn}. We let $I(\X,G_n)$ be the
intersection number of $\X$ with $G_n$, and let $I_0(\X,G_n)$ and
$I_1(\X,G_n)$ be the contributions to $I(\X,G_n)$ from the
$0$-dimensional components and the $1$-dimensional components of the
intersection, respectively. We also recall from Theorem \ref{theorem:
  correspondence} that $\X\simeq X_0^D(N)/W$, where
$$
W=W_Q:=\begin{cases}
  \gen{w_{DN}}, &\text{if }Q\text{ is not ambiguous},\\
  \gen{w_m,w_{DN}}, &\text{if }Q\text{ is ambiguous,
    representing }m\text{ or }4m,~m|DN. \end{cases}
$$

\subsection{Intersection numbers $I(\X,G_n)$ and Cohen's Eisenstein
  series of weight $5/2$ }

We first determine $I(\X,G_n)$. Let $a_n$ be defined by
\begin{align*}
\mathcal{H}(\tau)=\sum^\infty_{n=0}a_nq^n=1-10q-70q^4-48q^5-120q^8-250q^9-240q^{12}-\cdots.
\end{align*}
Note that $\cH(z)/120$ is Cohen's Eisenstein series of weight $5/2$.
We recall the following important result of
van der Geer.

\begin{theorem}[{\cite[Corollary
    8.2]{Geer-Siegel-threefold}}] \label{Geer theorem}
  The Humbert divisor $G_n$ is the zero divisor of a Siegel
  modular form of weight $-a_n/2$ .
\end{theorem}

From van der Geer's theorem, we easily obtain a formula for
$I(\X,G_n)$.

\begin{Proposition}\label{proposition: I}
  For all positive integer $n$ congruent to $0$ or $1$ modulo
  $4$, we have
  $$
  I(\X,G_n)=\frac1{|W|}a_nH_{D,N}(0),
  $$
  where $H_{D,N}(0)$ is defined by \eqref{equation: HDN(0)}.
  Consequently, one has
  $$
  \frac1{|W|}H_{D,N}(0)+\sum_{n=1}^\infty
  I(\X,G_n)e^{2\pi inz}=\frac1{|W|}H_{D,N}(0)\cH(z).
  $$
\end{Proposition}

\begin{proof}
  By Theorem 8.1 of \cite{Geer-Siegel-threefold}, there is a constant
  $c$ (depending on $\X$) such that $I(\X,G_n)=ca_n$ for all $n$.
  Thus, one only needs to know the value of one particular
  $I(\X,G_n)$. Here we let $n$ be a fundamental discriminant not
  represented by $Q$ so that $\X$ does not lie on $G_n$. 
  By Theorem
  \ref{Geer theorem}, $G_n$ is the divisor of a Siegel modular form 
  $F$ of weight $-a_n/2$ on $\Sp(4,\Z)$. Since, in general, the
  restriction of a Siegel modular form $f$ of weight $k$ along $\X$
  (that is, for $z\in\H$, define $\wt f(z)$ by $\wt f(z)=f(\tau_z)$,
  where $\tau_z$ is given by \eqref{equation: Omega_z 1} or
  \eqref{equation: Omega_z 2}) is a modular form of weight $2k$ on
  $X_0^D(N)/W$, the restriction of $F$ along $\X$ has
  $$
  \frac{|a_n|DN}{12|W|}\prod_{p|D}\left(1-\frac1p\right)
  \prod_{p|N}\left(1+\frac1p\right)=\frac1{|W|}a_nH_{D,N}(0)
  $$
  zeros. Hence $c=H_{D,N}(0)/\abs W$ and $I(\X,G_n)=a_nH_{D,N}(0)/|W|$. Then the proposition
  follows.
\end{proof}

We next determine $I_0(\X,G_n)$ in terms of class numbers. This is the
most complicated part of the proof.

\subsection{Determination of $I_0(\X,G_n)$}
By Lemma \ref{lemma: explicit mu perp}, we may assume that the
quadratic form $Q$ is $px^2+4sDNxy+4tDNy^2$ from Part (1) of the
lemma or $4px^2+4sDNxy+4tDNy^2$ from Part (2) of the lemma, depending
on whether $Q$ is primitive or of the form $4Q'$. Let $\O$
be the Eichler order of level $N$ in $\JS{-DN,p}\Q$ spanned by
\eqref{equation: Eichler basis 1} or \eqref{equation: Eichler basis
  2}. Let $\iota=\iota_Q:z\mapsto\tau_z$ be the corresponding
quaternion modular embedding from $\H$ to $\H_2$ given by
\eqref{equation: Omega_z 1} or \eqref{equation: Omega_z 2}. Let
$\mathcal F$ be the image of a fixed
fundamental domain of $X_0^D(N)$ under $\iota$. To avoid the
cumbersome factor $1/|W|$ and the complexity relating to fixed points
of Atkin-Lehner involutions, here we will work on $\F$ instead of
$\X$.

For a singular relation $\ell=(c_1,\ldots,c_5)\in\Z^5$, we let
$$
 \cH_\ell=\braces{\M{\tau_1}{\tau_2}{\tau_2}{\tau_3}\in\H_2:c_1\tau_1+c_2\tau_2+c_3\tau_3+c_4(\tau^2_2-\tau_1\tau_3)+c_5=0}
$$
be the locus of points in $\H_2$ satisfying $\ell$. We first recall
that if $\F\not\subset\cH_\ell$, then each point $\tau_z$
in $\cH_\ell\cap\F$ corresponds to an abelian surface with
endomorphism algebra strictly larger than quaternion algebra over
$\Q$, and hence $z$ is a CM-point on the Shimura curve and $\tau_z\in
\mathrm{M}(2,K)$ for some imaginary quadratic field $K$.

\begin{Lemma} Let $\ell$ be a singular relation such that
  $\F\not\subset\cH_\ell$. Then $\cH_\ell$ intersects $\F$
  transversally at each point of intersection.
\end{Lemma}

\begin{proof} Here we only prove the case $Q$ is of the form $4Q'$.

  Let $\ell=(c_1,\ldots,c_5)$ be a singular relation such that
  $\F\not\subset\cH_\ell$ and assume that
  $\cH_\ell$ intersects $\F$ at $\tau_z$, $z\in\H$. For convenience,
  we express $\SM{\tau_1}{\tau_2}{\tau_2}{\tau_3}$ as a row vector
  $(\tau_1,\tau_2,\tau_3)$. Using the explicit expression in
  \eqref{equation: Omega_z 2}, we find that the tangent space of $\F$
  at $\tau_z$ is spanned by
  \begin{equation} \label{equation: tangent 1}
    \left(DN+z^{-2},\frac{-DN+z^{-2}}{\sqrt p},
      \frac{DN+z^{-2}}p\right)
   =\frac1{\sqrt pz}(-p\tau_2,-\tau_1+1/2,-\tau_2).
  \end{equation}

  We next compute the tangent space of $\cH_\ell$.
  If $c_3\neq0$ or $c_4\neq0$, we may express $\tau_3$ as a function
  of $\tau_1$ and $\tau_2$. Then the tangent space of $\cH_\ell$ at
  $(\tau_1,\tau_2,\tau_3)$ is spanned by
  \begin{equation} \label{equation: tangent 2}
    (c_4\tau_1-c_3,0,c_1-c_4\tau_3), \quad
    (0,c_4\tau_1-c_3,2c_4\tau_2+c_2).
  \end{equation}
  The three vectors in \eqref{equation: tangent 1} and
  \eqref{equation: tangent 2} are coplanar if and only if
  $$
  2c_2\tau_1+2(pc_1+c_3-c_4)\tau_2+2c_4(\tau_1-p\tau_3)\tau_2-c_2=0.
  $$
  Since for $\tau_z$, we have $\tau_1-p\tau_3=(1+sDN)/2$, the
  condition reduces to
  \begin{equation} \label{equation: coplanar}
  2c_2\tau_1+(2pc_1+2c_3+(sDN-1)c_4)\tau_2-c_2=0.
  \end{equation}
  Now since $\tau_z\in \mathrm{M}(2,K)$ for some imaginary quadratic field
  $K$, considering the real part and the imaginary part separately, we
  see that the $\Z$-module of quintuple of integers $(c_1,\ldots,c_5)$
  satisfying both \eqref{equation: coplanar} and
  $c_1\tau_1+c_2\tau_2+c_3\tau_3+c_4(\tau_2^2-\tau_1\tau_3)+c_5=0$
  has rank $2$. However, as
  $\F\subset\cH_{\ell_1},\cH_{\ell_2}$, this $\Z$-module has to be
  spanned precisely by $\ell_1$ and $\ell_2$, where $\ell_1$
  and $\ell_2$ are the singular relations in \eqref{equation: Q mu
    basis 2}. Therefore, \eqref{equation: coplanar} can never hold for
  any singular relation with $c_3,c_4\neq0$ not in the span of
  $\ell_1$ and $\ell_2$, i.e., $\cH_\ell$ intersects $\F$
  transversally at $\tau_z$.

  If $c_1,c_4\neq 0$, we may express $\tau_1$ as a function
  of $\tau_2$ and $\tau_3$. Then by a similar computation, we find
  that $\cH_\ell$ intersects $\F$ transversally at $\tau_z$.

  Finally, if $c_1=c_3=c_4=0$, then the singular relation is
  $c_2\tau_2+c_5=0$ and the tangent space of $\cH_\ell$ at $\tau_z$ is
  spanned by $(1,0,0)$ and $(0,0,1)$. These two vectors cannot be
  coplanar with the vector in \eqref{equation: tangent 1} since
  $-\tau_1+1/2$ is never $0$ in $\H_2$. This completes the proof of
  the lemma.
\end{proof}

\begin{Corollary}\label{corollary: I_0} We have
  \begin{equation}
    \label{equation: sum of 0-dimensional intersection}
   I_0(\X_\mu,G_n)=\frac{1}{2|W|}
   \sum_{\ell:\Delta(\ell)=n,~\F\not\subset\cH_\ell}
   \abs{\cH_\ell\cap\F)}',
  \end{equation}
  where $\abs{\cH_\ell\cap\F}'$ denotes the weighted
  cardinality of the set $\cH_\ell\cap\F$ with the weight
  of CM-points of discriminant $-4$ given by $1/2$, the weight of
  CM-points of discriminant $-3$ given by $1/3$, and the weight of
  CM-points of other discriminant given by $1$.
\end{Corollary}

Note that the appearance of the factor $1/2$ in \eqref{equation: sum
  of 0-dimensional intersection} is due
to the fact that $\ell$ and $-\ell$ define the same surface
$\cH_\ell$.

Let $\ell_1,\ell_2$ be the singular relations given by
\eqref{equation: Q mu basis 1} or \eqref{equation: Q mu basis 2}
satisfied by all points $\tau_z$ in $\F$. Observe that if $\ell$ is a
singular relation of discriminant $n$ such that
$\F\not\subset\cH_\ell$, then the Gram matrix for the lattice spanned
by $\ell_1$, $\ell_2$, and $\ell$ is of the form
\begin{equation}\label{equation: M ell}
  M_\ell:=\begin{pmatrix}
    a&b&\gen{\ell_1,\ell}_\Delta \\
    b&c&\gen{\ell_2,\ell}_\Delta \\
    \gen{\ell_1,\ell}_\Delta&\gen{\ell_2,\ell}_\Delta&n\end{pmatrix},
\end{equation} 
where $\gen{\cdot,\cdot}_\Delta$ is the inner product defined in
\eqref{equation: inner P} and $\SM abbc$ is the Gram matrix of $\gen{\cdot,\cdot}_\Delta$ with
respect to $\ell_1$ and $\ell_2$. Here we shall partition the sum in
\eqref{equation: sum of 0-dimensional intersection} according to the
values of $\gen{\ell_1,\ell}_\Delta$ and $\gen{\ell_2,\ell}_\Delta$.
 
For convenience, given arbitrary integers $u$ and $v$, we let
\begin{equation}\label{equation: M uv}
 M_{u,v}=\begin{pmatrix}
 a&b& u\\
 b&c&v\\
 u&v&n
\end{pmatrix}.
\end{equation}
We have 
\begin{equation}\label{equation: sum of 0-dimensional intersection by Muv}
  I_0(\X_\mu,G_n)=\frac{1}{2|W|}\sum_{u,v\in\Z}
  \sum_{\ell:~\F\not\subset\cH_\ell,\
    M_\ell=M_{u,v}}\abs{\cH_\ell\cap\F}'.
\end{equation}
In order for $\cH_\ell\cap\mathcal{F}$ to be nonempty, the integer $u$
and $v$ need to satisfy
 \begin{enumerate}
 \item[(1)]$u\equiv an\mod 2$ and $v\equiv cn\mod 2$, due to
   \eqref{congruence relation}, 
 \item[(2)]$\det M_{u,v}=4\abs{d}$ for some negative discriminant
   $d$ such that $\Q(\sqrt{d})$ can be embedded in the quaternion
   algebra $B$, due to Lemma \ref{lemma: discriminant of CM
     lattice}.
 \end{enumerate}
Now assume that $u$ and $v$ are integers satisfying these
conditions. Write $d=f^2d_0$, where $d_0$ is a negative fundamental
discriminant. We need to study for each divisor $r$ of $f$, how many
CM-points $\tau_z$ of discriminant $r^2d_0$ in $\F$ satisfy some
singular relations $\ell$ with $M_\ell=M_{u,v}$ and how many such
$\ell$ they satisfy. We first mention a rather trivial but yet
important observation.

\begin{Lemma} \label{lemma: matrix equation}
  Let the notations be as above. For a CM-point $\tau_z$
  of discriminant $r^2d_0$, let $M_z$ be the Gram matrix for the
  lattice of singular relations satisfied by $\tau_z$ given in
  \eqref{equation: CM Gram1} or \eqref{equation: CM Gram2}. Then
  $\tau_z\in\cH_\ell$ for some $\ell$ with $M_\ell=M_{u,v}$ if and
  only if the equation
  \begin{equation} \label{equation: matrix equation}
  \begin{pmatrix}1&0&0\\0&1&0\\x&y&w\end{pmatrix}M_z
  \begin{pmatrix}1&0&x\\0&1&y\\0&0&w\end{pmatrix}=M_{u,v}
  \end{equation}
  in $x$, $y$, and $w$ is solvable in integers.

  Moreover, in the case
  the equation is solvable in integers, there is precisely one
  singular relation $\ell$ with $\tau_z\in\cH_\ell$ and
  $M_\ell=M_{u,v}$.
\end{Lemma}

\begin{proof} The first statement is obvious. Here the correspondence
  between a solution $(x,y,w)$ and a singular relation is
  $(x,y,w)\leftrightarrow x\ell_1+y\ell_2+w\ell_3$, where $\ell_3$ is
  the singular relation for $\tau_z$ described in Lemma \ref{lemma:
    discriminant of CM lattice}. The uniqueness of
  $\ell$ is due to the fact that the quadratic form on the lattice of
  singular relations satisfied by $\tau_z$ is nondegenerate.
\end{proof}


\begin{Proposition}
Let the notations be as above. (In particular, $d$ is the negative
discriminant such that $\det M_{u,v}=4|d|$.) The number of CM-points
of discriminant $r^2d_0$ satisfying some singular relation $\ell$ with
$M_\ell=M_{u,v}$ 
is
$$
2\cdot\frac{h_{D,N}(r^2d_0)}{2^{\omega_{DN}(d)}},
$$
where $h_{D,N}(r^2d_0)$ is the number of CM-points of discriminant
$r^2d_0$ on $\XDN$ and 
$$
\omega_{DN}(d)=\#\braces{q\ \mathrm{prime}:\ q\vert DN,\ q\nmid d}.
$$
\end{Proposition}

\begin{proof}
We only prove the case $Q$ is primitive. The proof of other case
$Q=4Q$ is similar.

Let $b_1\beta_1+b_2\beta_2+b_3\beta_3$ be the image of
$
\sqrt{r^2d_0}
$
under the normalized optimal embedding associated to a CM-point $z$
of discriminant $r^2d_0$, where $\beta_1,\beta_2,\beta_3$ are given in (\ref{equation: basis of trace zero}).
We recall from Lemma \ref{lemma: discriminant of CM lattice} that $\ell_1,\ \ell_2,$ and
$$
\ell_3=(0,b_2/2,b_3/2+b_2(1-p)/4,0,b_1)
$$
form a basis of singular relations for $\tau_z$, the Gram matrix
$M_z$ with respect to $\ell_1,\ell_2,\ell_3$ is
$$
M_z=\begin{pmatrix}
p&2sDN&-b_2p/2-b_3\\
2sDN&4tDN&2b_1-b_2sDN\\
-b_2p/2-b_3 &2b_1-b_2sDN &b^2_2/4
\end{pmatrix},
$$ and $\det M_z=4\abs{r^2d_0}$.
By Lemma \ref{lemma: matrix equation}, we need to determine whether
the equation in \eqref{equation: matrix equation} is solvable in
integers.

By considering the determinants of the two sides of \eqref{equation:
  matrix equation}, we see that $w=r'$ or $w=-r'$, where $r'=f/r$.
Then \eqref{equation: matrix equation} is solvable in integers if and
only if
\begin{align*}
&x=t(u+w(b_2p/2+b_3))-\frac{s}{2}(v-w(2b_1-b_2sDN)),\\
&y=-\frac{s}{2}(u+w(b_2p/2+b_3))+\frac{p}{4DN}(v-w(2b_1-b_2sDN))
\end{align*}
are integers. Since $s$ is assume to be even (see the assumptions in
Lemma \ref{lemma: explicit mu perp}), $x$ is always an integer, while
the condition for $y$ being an integer reduces to that
$$
w(2b_1-b_2sDN)\equiv v\mod 4DN,\quad w=\pm r'.
$$
Note that $b_2$ is even. Thus, the congruence equation further reduces
to
\begin{equation} \label{equation: congruence equation}
b_1w\equiv v/2\mod 2DN,\quad w=\pm r'.
\end{equation}
We observe that
$$
-4d=\det M_{u,v}=-pv^2+4sDNuv-4tDNu^2+4DNn.
$$
Since $t=(s^2DN+1)/p\equiv 1\mod 4$ and $u^2\equiv n\mod 4$ by
\eqref{congruence relation}, we have
\begin{equation} \label{equation: v}
p(v/2)^2\equiv d \mod 4DN.
\end{equation}
Also, we recall from Part (1) of Lemma \ref{lemma: Galois orbits} that
\begin{equation} \label{equation: b1}
pb_1^2\equiv r^2d_0\mod 4DN.
\end{equation}
Naturally, we will consider the congruence equation \eqref{equation:
  congruence equation} locally.

For a prime divisor $q$ of $DN$, let $S(\O,r^2d_0,q)$ be the set of
$\O_q^\times$-inequivalent optimal embeddings of discriminant $r^2d_0$
into $\O_q:=\O\otimes\Z_q$. Its cardinality $m(\O,r^2d_0,q)$
is given by \eqref{equation: number of local optimal
  embeddings}. For an optimal embedding in $S(\O,r^2d_0,q)$, write
again the image of $\sqrt{r^2d_0}$ under this embedding as
$b_1\beta_1+b_2\beta_2+b_3\beta_3$, $b_j\in\Z_q$.

Consider the case $w=r'$ first. Let $T(\O,r^2d_0,q)$
be the subset of $S(\O,r^2d_0,q)$ such that \eqref{equation:
  congruence equation} holds at $q$, i.e.,
\begin{equation} \label{equation: local b1}
b_1r'\equiv v/2\mod \begin{cases}
  4, &\text{if }q=2, \\
  q, &\text{if }q\neq 2. \end{cases}
\end{equation}
For the case $q$ is a prime dividing $d$ (and also $DN$), we
have by \eqref{equation: v} and \eqref{equation: b1},
$$
v\equiv 0\mod\begin{cases}
  4, &\text{if }q=2, \\
  q, &\text{if }q\neq 2,\end{cases}
$$
and
\begin{equation} \label{equation: bw}
pb_1^2(r')^2\equiv (rr')^2d_0=d\equiv 0\mod
\begin{cases}
  8, &\text{if }q=2, \\
  q, &\text{if }q\neq 2.\end{cases}
\end{equation}
Thus, \eqref{equation: local b1} holds for every element in
$S(\O,r^2d_0,q)$. In other words, for the case $q|d$, we have
$|T(\O,r^2d_0,q)|=m(\O,r^2d_0,q)$.

For the case $q$ does not divide $d$ (but divides $DN$), we have $m(\O,r^2d_0,q)=2$ by \eqref{equation: number of local optimal embeddings}, and according to Part (2) of Lemma \ref{lemma:
  Galois orbits}, the equivalence of local optimal embeddings is
completely determined by the residue class of $b_1$ modulo $4$ if
$q=2$, or of $b_1$ modulo $q$ if $q$ is odd. The computation
\eqref{equation: bw} above and \eqref{equation: v} show that both
elements in $S(\O,r^2d_0,q)$ satisfy
$$
p(b_1r')^2\equiv p(v/2)^2\mod\begin{cases}
  8,&\text{if }q=2, \\
  q,&\text{if }q\neq 2.\end{cases}
$$
However, for the given $v$, exactly one of them will satisfy
\eqref{equation: local b1}. Therefore, in the case $q$ does not divide
$d$, we have $|T(\O,r^2d_0,q)|=m(\O,r^2d_0,q)/2$.

In summary, we find that the total number of CM-points of discriminant
$r^2d_0$ such that their corresponding $b_1$ satisfy \eqref{equation:
  congruence equation} with $w=r'$ is
$$
h(r^2d_0)\prod_{q|(DN,d)}m(\O,r^2d_0,q)
\prod_{q|DN,~q\nmid d}\frac{m(\O,r^2d_0,q)}2
=\frac{h_{D,N}(r^2d_0)}{2^{\omega_{DN}(d)}},
$$
where $h_{D,N}(r^2d_0)$ denotes the number of CM-points of
discriminant $r^2d_0$ on $X_0^D(N)$. It is clear that the case $w=-r'$
give the same number of CM-points, and the proof of the proposition is
complete.
\end{proof}

Putting \eqref{equation: sum of 0-dimensional intersection by Muv},
Lemma \ref{lemma: matrix equation}, and the above proposition
together, and summing over all positive divisors $r$ of $f$, we obtain
$$
I_0(\X,G_n)=\frac1{|W|}\sum_{u\equiv an,~v\equiv cn\mod 2,~\det M_{u,v}>0}
H_{D,N}(\det M_{u,v}/4).
$$
Since $\det M_{u,v}=4DNn-Q(v,-u)$, we arrive at the following formula
for $I_0(\X,G_n)$.

\begin{Proposition} \label{proposition: I0}
We have
  $$
  I_0(\X,G_n)=\frac1{|W|}\sum_{u\equiv an,~v\equiv cn\mod
    2,~Q(v,u)<4DNn}H_{D,N}\left(DNn-\frac{Q(v,u)}4\right).
  $$
\end{Proposition}

\subsection{Determination of $I_1(\X,G_n)$}

As above, we let $\F$ be the image of a fundamental domain of
$X_0^D(N)$ under the quaternion modular embedding given by
\eqref{equation: Omega_z 1} or \eqref{equation: Omega_z 2}, according
to whether $Q$ is primitive or is of the form $4Q'$. We have
$$
I_1(\X,G_n)=\frac1{2|W|}\Vol(\F)
\sum_{\ell:~\Delta(\ell)=n,~\F\subset\cH_\ell}1.
$$
Here again, the appearance of the factor $1/2$ is due to the fact that
$\cH_\ell$ and $\cH_{-\ell}$ define the same surface in $\H_2$.
Note that $\F\subset\cH_\ell$ if and only if $\ell$ is in the span of
$\ell_1$ and $\ell_2$, i.e., if and only if the matrix $M_\ell$ in
\eqref{equation: M ell} has determinant $0$. Also, such a singular
relation is completely determined by the values of
$\gen{\ell_1,\ell}_\Delta$ and $\gen{\ell_2,\ell}_\Delta$. It follows
that
$$
I_1(\X,G_n)=\frac1{2|W|}\Vol(\F)\sum_{u\equiv an,~v\equiv cn\mod 2,
  ~\det M_{u,v}=0}1,
$$
where $M_{u,v}$ is the matrix in \eqref{equation: M uv}.
Since $\det M_{u,v}=4DNn-Q(v,-u)$ and
$$
\frac12\Vol(\F)=-\frac{DN}{12}\prod_{q|D}\left(1-\frac1q\right)
\prod_{q|N}\left(1+\frac1q\right)=H_{D,N}(0),
$$
we may write the formula as follows.

\begin{Proposition} \label{proposition: I1} We have
  $$
  I_1(\X,G_n)=\frac1{|W|}\sum_{u\equiv an,~v\equiv cn\mod 2,
    ~Q(v,u)=4DNn}H_{D,N}\left(DNn-\frac{Q(v,u)}4\right).
  $$
\end{Proposition}

\subsection{Proof of Theorem \ref{theorem: class number relation}}

Combining Propositions \ref{proposition: I}, \ref{proposition: I0},
and \ref{proposition: I1}, we obtain the class number relations
\eqref{equation: new class number relation}.

\bibliographystyle{plain}
\bibliography{ShimuraCurves}

\end{document}